\documentclass[12pt, a4paper]{amsart}

\usepackage[hmargin=30mm, vmargin=25mm, includefoot, twoside]{geometry}
\usepackage[bookmarksopen=true]{hyperref}

\usepackage{amsfonts,amssymb,verbatim}

\usepackage{latexsym}
\usepackage{mathrsfs}
\usepackage{stmaryrd}
\usepackage{xspace}
\usepackage{enumerate, paralist}
\usepackage{graphicx}
\usepackage[all]{xy}
\usepackage{extarrows}
\usepackage{tabu}
\usepackage{accents}
\usepackage{tensor}

\usepackage[usenames,dvipsnames]{color}

\usepackage{txfonts, pxfonts}

\usepackage{amsmath}

\usepackage[tbtags]{mathtools}

\usepackage{amsthm, thmtools}

\newtheorem{thm}{Theorem}[section]

 \newtheorem{cor}[thm]{Corollary}
 \newtheorem{lem}[thm]{Lemma}
 \newtheorem{prop}[thm]{Proposition}

\newtheorem{introthm}{Theorem}

\numberwithin{equation}{section}

 \theoremstyle{definition}

 \theoremstyle{remark}

\newtheorem*{claim*}{Claim}

\def\NN{\mathbb{N}}
\def\CC{\mathbb{C}}

\def\B{\mathfrak{B}}
\def\R{\mathfrak{R}}
\def\P{\mathfrak{P}}
\def\M{\mathfrak{M}}

\def\H{\mathcal{H}}

\def\N{\mathcal{N}}

\def\SOT{\mathrm{(SOT)}}
\def\Id{\mathrm{Id}}

\def\ppg{\mathrm{prop}}
\def\ran{\mathrm{Ran}}
\def\ker{\mathrm{Ker}}
\def\dim{\mathrm{dim}}
\def\Ad{\mathrm{Ad}}
\def\rank{\mathrm{rank}}

\def\m{\mathbf{m}}

\begin{document}

\title{Quantitative decomposition and approximation for quasi-local operators}
\author{Jiawen Zhang}
\author{Jingming Zhu}

\address[Jiawen Zhang]{School of Mathematical Sciences, Fudan University, 220 Handan Road, Shanghai, 200433, China.}
\email{jiawenzhang@fudan.edu.cn}

\address[Jingming Zhu]{College of Data Science, Jiaxing University, 899 Guangqiong Road, Jiaxing, 314000, China.}
\email{jingmingzhu@zjxu.edu.cn}

\thanks{Jiawen Zhang was partly supported by the National Key R{\&}D Program of China 2022YFA100700 and NSFC 12422107.}


\begin{abstract}
We develop a quantitative approach to quasi-local operators with bounded block-rank. Our main results are a quantitative decomposition for quasi-local operators into block-rank-one pieces and a quantitative finite-propagation approximation in the block-rank-one case, leading to a quantitative proof of the bounded block-rank rigidity result for Roe and quasi-local algebras.
\end{abstract}

\date{\today}

\maketitle

\section{Introduction}\label{sec:intro}

The Roe algebra and the quasi-local algebra are $C^*$-algebras associated to metric spaces, introduced by John Roe in \cite{Roe88}. They are central objects in higher index theory since their $K$-theories contain information about higher indices of elliptic (pseudo-)differential operators on open manifolds (\cite{Roe93, Roe96}). These algebras also play an important role in the study of operator algebras, coarse geometry and even mathematical physics (\emph{e.g.}, \cite{BBFKVW22, BFV20, EM19, MV25, WW20}).

Central questions are when the Roe algebra coincides with the quasi-local algebra, and when the inclusion between them induces an isomorphism on their $K$-theories, as stated in \cite[Question 1.1]{LZZ26}. These questions have drawn considerable attention over the last few years. 

To be more precise, let us first explain the notions. For a discrete metric space $(X,d)$, denote $N_X(R):=\sup_{x\in X} |B(x,R)|$ for $R \geq 0$, where $B(x,R)$ is the closed ball. Say that $X$ has \emph{bounded geometry} if $N_X(R) < \infty$ for each $R \geq 0$. Given a Hilbert space $\H_0$ and $T \in \B(\ell^2(X;\H_0))$, write $T=(T_{x,y})_{x,y\in X}$ for $T_{x,y} \in \B(\H_0)$ given by 
\[
\langle T_{x,y}\xi,\eta \rangle = \langle T(\delta_y \otimes \xi), \delta_x \otimes \eta \rangle \quad \text{for} \quad \xi, \eta \in \H_0. 
\]
\begin{itemize}
 \item Say that $T$ has \emph{finite propagation} if $\ppg(T):=\sup\{d(x,y): T_{x,y} \neq 0\} < \infty$, and $T$ is \emph{locally compact} if $T_{x,y}$ is compact for each $x,y\in X$;
 \item Given $\varepsilon, R>0$, say that $T$ is \emph{$(\varepsilon,R)$-quasi-local} if for any $A,B \subseteq X$ with $d(A,B) > R$, $\|\chi_A T \chi_B\| < \varepsilon$. We say that $T$ is \emph{quasi-local} if for any $\varepsilon>0$, there exists $R>0$ such that $T$ is $(\varepsilon,R)$-quasi-local.
\end{itemize}
The \emph{Roe algebra} $C^*(X;\H_0)$ is the norm closure of all finite propagation and locally compact operators on $\ell^2(X;\H_0)$, and the \emph{quasi-local algebra} $C^*_q(X;\H_0)$ is the set of all quasi-local and locally compact operators on $\ell^2(X;\H_0)$.

In \cite{SZ20}, \v{S}pakula and the first-named author showed that if $X$ has Yu's Property A (introduced in \cite{Yu00} and simplified in \cite{NWZ25, ZZ25}), then the Roe and quasi-local algebras are the same. On the other hand, Ozawa recently showed in \cite{Oza25} that if $X$ contains a sequence of asymptotic expanders (introduced in \cite{LNSZ21}), then these two algebras are not the same. 

Here we focus on coarse disjoint unions of finite metric spaces. Let $\{(X_n,d_n)\}_{n\in \NN}$ be a sequence of finite metric spaces. Their \emph{coarse disjoint union} is a metric space $(X,d)$ where $X = \bigsqcup_{n\in \NN} X_n$ as a set, the restriction of $d$ on each $X_n$ is $d_n$ and $d$ satisfies
\[
d(X_n, X_m) \to \infty \quad \text{as} \quad n+m \to \infty \quad \text{and} \quad n \neq m.
\]
Clearly, $X$ has bounded geometry if and only if $\{(X_n,d_n)\}_n$ has \emph{uniformly bounded geometry} (\emph{i.e.}, $\sup_n N_{X_n}(R) < \infty$ for each $R \geq 0$). An operator $T\in \B(\ell^2(X;\H_0))$ is called \emph{block-diagonal} if $T$ has the form $T=\SOT$-$\sum_{n\in \NN} T_n$ for $T_n \in \B(\ell^2(X_n;\H_0))$. Given $M \in \NN$, such $T$ is said to have \emph{block-rank at most $M$} if each $T_n$ has finite rank at most $M$, and \emph{block-rank $M$} if each $T_n$ has rank exactly $M$. 

To attack the $K$-theory question, the first-named author and his collaborators have studied block-rank-one projections $P=\mathrm{(SOT)}$-$\sum_{n} P_n$ on sparse spaces in \cite{KLVZ21, LSZ23}, and showed that $P$ is quasi-local if and only if it is in the Roe algebra. Recently, this result has been generalised to the higher block-rank case:

\begin{introthm}[{\cite[Theorem A]{LZZ26}}]\label{thm:Thm A in LZZ26}
 Let $\{(X_n,d_n)\}_{n\in \NN}$ be a sequence of finite metric spaces of uniformly bounded geometry, and $X:=\bigsqcup_{n\in \NN} X_n$ be their coarse disjoint union. Let $M \in \NN$ and $\H_0$ be a Hilbert space. For $T=\SOT$-$\sum_{n\in \NN} T_n$ in $\B(\ell^2(X;\H_0))$ with block-rank at most $M$, we have $T \in C^*(X;\H_0)$ \emph{if and only if} $T \in C^*_q(X;\H_0)$.
\end{introthm}

The key ingredient in the proof of Theorem~\ref{thm:Thm A in LZZ26} is a decomposition result \cite[Theorem B]{LZZ26}, showing that such $T$ can be decomposed as a sum of $M$ quasi-local block-diagonal operators with block-rank at most one. Based on these works, \cite[Theorem D]{LZZ26} provides the first negative result on their $K$-theories.

In this paper, we provide a quantitative proof of Theorem~\ref{thm:Thm A in LZZ26}. First, we establish the following quantitative decomposition result in Section~\ref{sec:decomposition} (see Theorem~\ref{thm:quantitative decomposition general case} for a more precise statement):

\begin{introthm}\label{thm:quantitative decomposition intro}
Let $\{(X_n,d_n)\}_{n\in\NN}$ be a sequence of finite metric spaces and let $X=\bigsqcup_n X_n$ be a coarse disjoint union. Let $\H_0$ be a Hilbert space and $M \in \NN$. Then there exist $s_M, \alpha_M \in (0,1)$ and $A_M, B_M \geq 1$ such that for any $(\varepsilon,R)$-quasi-local block-diagonal contraction $T \in \B(\ell^2(X;\H_0))$ with block-rank at most $M$ and $\varepsilon\in (0,s_M)$, there exists a decomposition $T=T^{(1)}+\cdots+T^{(M)}$, where each $T^{(j)}$ is a block-diagonal contraction with block-rank at most one and is $(A_M\varepsilon^{\alpha_M},B_MR)$-quasi-local.
\end{introthm}

Theorem~\ref{thm:quantitative decomposition intro} differs from \cite[Theorem B]{LZZ26}, where \(T\) is assumed to be quasi-local and each \(T^{(j)}\) can be chosen to be quasi-local. In the current setting, \(T\) need not be quasi-local and instead, we assume only a single quantitative quasi-locality estimate and obtain quantitative estimates for the summands. Hence, neither of these two results contains the other. To obtain the required quantitative estimates, we need a different approach rather than simply tracking the parameters in the proof of \cite[Theorem B]{LZZ26}.

Based on Theorem~\ref{thm:quantitative decomposition intro}, it remains to approximate $(\varepsilon',R')$-quasi-local block-rank-one operators by finite propagation operators. We prove the following in Section~\ref{sec: rank one} (see Theorem~\ref{thm:quantitative version of rank one general case} for a precise statement).

\begin{introthm}\label{thm:quantitative version of rank one intro}
Let $\{(X_n,d_n)\}_{n\in\NN}$ be a sequence of finite metric spaces with uniformly bounded geometry and let $X=\bigsqcup_n X_n$ be a coarse disjoint union. Let $\H_0$ be a Hilbert space, and $T=(\mathrm{SOT})\text{-}\sum_n T_n$ be a block-rank-one contraction in $\B(\ell^2(X;\H_0))$ which is $(\varepsilon,R)$-quasi-local for some $\varepsilon \in (0,\frac{1}{30})$ and $R>0$. Then there exists a locally compact operator $S \in \B(\ell^2(X;\H_0))$ with propagation at most $R'$, where $R'$ depends only on $\varepsilon,R$ and the uniformly bounded geometry of $\{X_n\}_n$, such that $\|T-S\| \leq 30\varepsilon$.
\end{introthm}

When each $T_n$ is a projection, this is a quantitative version of \cite[Theorem 6.6]{LSZ23} by parameter tracing. However, the original proof of \cite[Theorem 6.6]{LSZ23} is technical and hence, for the convenience of the reader, we provide a self-contained proof and compute the parameters in Section~\ref{sec:proof of complex case rank one}.

Finally, combining Theorem~\ref{thm:quantitative decomposition intro} and Theorem~\ref{thm:quantitative version of rank one intro}, we recover Theorem~\ref{thm:Thm A in LZZ26}.

\subsection*{Acknowledgements}
We would like to thank Kang Li for helpful discussions.

\subsection*{Notation} Given a metric space $(X,d)$, $A \subseteq X$ and $R \geq 0$, denote its \emph{$R$-neighbourhood} by $\N_R(A):=\{x\in X: d(x,A) \leq R\}$, and its \emph{$R$-boundary} by $\partial_R(A):=\N_R(A) \setminus A$. Moreover, for $A \subseteq Y \subseteq X$, denote the relative boundary by $\partial^Y_R(A):=\partial_R(A) \cap Y$. For $E \subseteq X \times X$, denote $E^{-1}:=\{(y,x) \in X \times X: (x,y) \in E\}$.

A subset $K \subseteq X \times X$ is called a \emph{partial translation} if there exists $R>0$ such that $K \subseteq \{(x,y) \in X \times X: d(x,y) \leq R\}$ and for each $x\in X$, there is at most one $y\in X$ with $(x,y) \in K$ and at most one $z\in X$ with $(z,x) \in K$. 

\section{Quantitative decomposition}\label{sec:decomposition}

Let $(X,d)$ be a discrete metric space, $\H_0$ be a Hilbert space and $u^{(1)}, \cdots, u^{(M)}$ be a set of orthonormal vectors  in $\ell^2(X;\H_0)$. Denote by $U$ the subspace in $\ell^2(X;\H_0)$ spanned by $\{u^{(1)}, \cdots, u^{(M)}\}$, and $P_U$ the orthogonal projection onto $U$. For simplicity, write $u_A:=\chi_A u$ for $u\in \ell^2(X;\H_0)$ and $A \subseteq X$. 

Set $W: \CC^M \longrightarrow \ell^2(X;\H_0)$ by $W(e_i):=u^{(i)}$ for $i=1,2,\cdots, M$. Here \(e_i\) denotes the \(i\)-th standard basis vector of \(\CC^M\). Clearly, 
\[
W^*W = \Id_{\CC^M} \quad \text{and} \quad WW^* = P_U.
\]
For $A \subseteq X$, denote 
\[
\m(A):=W^* \chi_A W \in \M_M(\CC),
\]
where $\M_M(\CC)$ denotes the $M$-by-$M$ complex matrix algebra. We have:
\begin{equation}\label{EQ:u and m}
\|u_A\|^2 = \langle \m(A) \xi, \xi \rangle = \|\m(A)^{1/2} \xi\|^2  \quad \text{for} \quad u=W \xi.
\end{equation}

Moreover, we recall the following from \cite[Lemma 3.1]{LZZ26}.

\begin{lem}\label{lem:norm of APB}
 With the same notation as above, we have:
 \[
   \|\chi_A P_U \chi_B\| = \left\|\m(A)^{1/2} \m(B)^{1/2}\right\| \quad \text{for any} \quad A,B \subseteq X.
 \]
\end{lem}

We need the following two auxiliary results.

\begin{prop}\label{prop:splitting}
With the same notation as above, put $\varepsilon=\|\chi_AP_U\chi_B\|$. Then there is an orthogonal decomposition $U=U_A\oplus U_B$, where $U_A, U_B$ are subspaces, such that
\[
\|\chi_AP_{U_A}\|\le \sqrt \varepsilon,
\qquad
\|\chi_BP_{U_B}\|\le \sqrt \varepsilon.
\]
\end{prop}

\begin{proof}
Set $e_1:= \chi_{[0,\varepsilon]}(\m(A))$ and $e_2:=\chi_{(\varepsilon,1]}(\m(A))$, which are the spectral projections of the matrix $\m(A)$ with $e_1 + e_2 = \Id_M$.  Define
\[
   U_A:= W( e_1 \CC^M)
   \quad \text{and} \quad
   U_B:= W (e_2 \CC^M).
\]
Then we have the decomposition $U=U_A\oplus U_B$. 

First, assume that $\varepsilon>0$. On $e_1 \CC^M$, we have $\m(A)\le \varepsilon \Id_M$. Hence for $\xi\in e_1 \CC^M$, $\|\chi_A W\xi\|^2 =\langle \m(A)\xi,\xi\rangle\leq \varepsilon\|\xi\|^2$, which implies that $\|\chi_AP_{U_A}\|\leq \sqrt \varepsilon$.
On the other hand, we have $\m(A)\geq \varepsilon \Id_M$ on $e_2 \CC^M$. Set $T:=\m(A)|_{e_2 \CC^M}$. Then $T$ is invertible and $\|T^{-1/2}\| \leq \varepsilon^{-1/2}$. For $\xi \in e_2 \CC^M$, we have $\xi=T^{1/2}T^{-1/2}\xi=\m(A)^{1/2}T^{-1/2}\xi$.
So
\begin{align*}
\|\chi_B W\xi\|
&=\|\m(B)^{1/2}\xi\|=\left\| \m(B)^{1/2} \m(A)^{1/2}T^{-1/2}\xi \right\|\\
&\le \|\m(B)^{1/2}\m(A)^{1/2}\| \cdot \|T^{-1/2}\|\cdot \|\xi\| \leq \varepsilon \cdot \varepsilon^{-1/2} \cdot \|\xi\| =\sqrt \varepsilon \cdot \|\xi\|,
\end{align*}
where we use \eqref{EQ:u and m} in the first equality and Lemma~\ref{lem:norm of APB} in the penultimate inequality. Hence $\|\chi_B P_{U_B}\| \leq \sqrt \varepsilon$. 

Finally for $\varepsilon=0$, note that $\m(A)$ vanishes on $e_1 \CC^M$ and hence by \eqref{EQ:u and m}, $\chi_AP_{U_A} = 0$. Moreover, Lemma~\ref{lem:norm of APB} shows that $\m(B)^{1/2}\m(A)^{1/2}=0$, which means that $\m(B)^{1/2}$ vanishes on $\ran(\m(A)^{1/2}) = e_2 \CC^M$. This gives $\chi_B P_{U_B} = 0$ by \eqref{EQ:u and m} again. 
\end{proof}

\begin{lem}\label{lem:large portion of norm}
 Let $(X,d)$ be a discrete metric space, $\H_0$ a Hilbert space, and $u \in \ell^2(X;\H_0)$ a unit vector. For $A \subseteq X$, $\varepsilon, \varepsilon'\in (0,1)$ and $R>0$, if $\chi_A P_u$ is $(\varepsilon,R)$-quasi-local and $\|u_A\| \geq \varepsilon'$, then $P_u$ is $(\sqrt{2}\varepsilon/\varepsilon', 2R)$-quasi-local.
\end{lem}

\begin{proof}
 If not, there would exist $C,D \subseteq X$ with $d(C,D) > 2R$ such that $\|\chi_C P_u \chi_D\| \geq \sqrt{2}\varepsilon/\varepsilon'$. This implies that $\|u_C\| \geq \sqrt{2}\varepsilon/\varepsilon'$ and $\|u_D\| \geq \sqrt{2}\varepsilon/\varepsilon'$.

 Denote $A_1:=A \cap \N_R(C)$ and $A_2:=A \setminus A_1$. Since $\|u_A\| \geq \varepsilon'$, either $\|u_{A_1}\| \geq \varepsilon'/\sqrt{2}$ or $\|u_{A_2}\| \geq \varepsilon'/\sqrt{2}$. Assume the former holds, and the latter is similar. Then
 \begin{equation}\label{EQ:lem:large portion}
   \|\chi_{A_1} P_u \chi_D\| = \|u_{A_1}\| \cdot \|u_D\| \geq \frac{\varepsilon'}{\sqrt{2}} \cdot \frac{\sqrt{2}\varepsilon}{\varepsilon'} = \varepsilon.
 \end{equation}
 On the other hand, the assumption that $\chi_A P_u$ is $(\varepsilon,R)$-quasi-local shows that
 \[
   \|\chi_{A_1} P_u \chi_D\| = \|\chi_{A_1} (\chi_A P_u)\chi_D\| < \varepsilon
 \]
 since $d(A_1,D) > R$. This contradicts (\ref{EQ:lem:large portion}).
\end{proof}

Now we are ready to prove the quantitative decomposition:

\begin{thm}\label{thm:quantitative decomposition}
For each $M \in \NN$, there exist constants $t_M, \beta_M \in (0,1)$ and $C_M, D_M \geq 1$, depending only on $M$, such that the following holds: 
Let $(X,d)$ be a discrete metric space, $Y\subseteq X$, $H_0$ a Hilbert space, and let $U \subseteq \ell^2(X;\H_0)$ be an $M$-dimensional subspace. If $\chi_Y P_U$ is $(\varepsilon, R)$-quasi-local for some $\varepsilon \in (0,t_M)$ and $R>0$, then $U$ has an orthonormal basis $\{u^{(1)}, \cdots, u^{(M)}\}$ such that each $\chi_Y P_{u^{(i)}}$ is $(C_M \varepsilon^{\beta_M}, D_M R)$-quasi-local. 
\end{thm}

\begin{proof}
 We argue by induction on $M$, and the case $M=1$ holds trivially. Assume the result holds for every dimension less than $M$. Since there are only finitely many such dimensions, we may choose parameters $t,\beta\in (0,1)$ and $C,D\geq 1$ such that the conclusion holds with $t,\beta,C,D$ for every dimension less than $M$.
 
We now prove the statement in dimension $M$. Denote $K:=2\sqrt{2M}\,C$ and we set
\[
t_M:=\min\left\{\frac{t^2}{4}, \left(\frac{t}{K+1}\right)^{4/\beta} \right\}, \quad \beta_M:=\frac{\beta^2}{4}, \quad C_M:=C(K+1) \quad \text{and} \quad D_M:=2D^2.
\]
We claim that these parameters work.
Fix a discrete metric space $(X,d)$, a Hilbert space $\H_0$, a subset $Y \subseteq X$ and an $M$-dimensional subspace $U \subseteq \ell^2(X;\H_0)$. Assume that $\chi_Y P_U$ is $(\varepsilon, R)$-quasi-local for some $\varepsilon \in (0,t_M)$.

Denote $\Delta:=2\varepsilon^{\beta/4}$. Since $\sqrt\varepsilon \leq \varepsilon^{\beta/4}$, we have 
\begin{equation}\label{EQ:Delta}
\Delta>\sqrt{\varepsilon} 
\quad \text{and} \quad
\Delta-\sqrt{\varepsilon}\ge \varepsilon^{\beta/4}.
\end{equation}
Take an arbitrary unit vector $u \in U$, and we divide into two cases.

 \noindent \textbf{Case I.} Suppose that $\chi_YP_u$ is $(\Delta,R)$-quasi-local. Since $\Delta=2\varepsilon^{\beta/4}\leq K\varepsilon^{\beta/4}$ and $D \geq 1$, $\chi_YP_{u^{(1)}}$ is $(K\varepsilon^{\beta/4},2DR)$-quasi-local for $u^{(1)}:=u$.

 \noindent \textbf{Case II.} Suppose that $\chi_YP_u$ is not $(\Delta,R)$-quasi-local. Then there exist $A \subseteq Y$ and $B \subseteq X$ with $d(A,B) > R$ such that $\|u_A\|, \|u_B\| \geq \Delta$. Since $\chi_YP_U$ is $(\varepsilon,R)$-quasi-local, we have $\|\chi_AP_U\chi_B\|<\varepsilon$. Applying Proposition~\ref{prop:splitting}, there is an orthogonal decomposition $U=U_A\oplus U_B$ such that
\begin{equation}\label{EQ:splitting in proof}
\|\chi_AP_{U_A}\|<\sqrt{\varepsilon}
\quad \text{and} \quad
\|\chi_BP_{U_B}\|<\sqrt{\varepsilon}.
\end{equation}
Write $u=a+b$ for $a\in U_A$ and $b \in U_B$. By \eqref{EQ:Delta} and \eqref{EQ:splitting in proof}, we obtain
\begin{equation}\label{EQ:lower bdd of bA}
\|b_A\| \geq \|u_A\|-\|a_A\| \geq \Delta-\sqrt\varepsilon \geq \varepsilon^{\beta/4}.
\end{equation}
Hence $U_B \neq 0$. Similarly, $U_A \neq 0$. Hence, $1\leq \dim(U_B) < M$.

Since $\chi_AP_{U} = \chi_A\chi_YP_{U}$ is $(\varepsilon,R)$-quasi-local and $\|\chi_AP_{U_A}\|<\sqrt\varepsilon$, we obtain that $\chi_AP_{U_B}$ is $(\varepsilon+\sqrt\varepsilon,R)$-quasi-local. Note that $\varepsilon+\sqrt\varepsilon
\leq 2\sqrt\varepsilon<t$ since $\varepsilon<t_M\leq \frac{t^2}{4}$. The inductive hypothesis applied to $\chi_AP_{U_B}$ provides an orthonormal basis $v^{(1)},\cdots,v^{(r)}$ of $U_B$ with $r=\dim U_B$ such that each $\chi_AP_{v^{(i)}}$ is $\left(C(\varepsilon+\sqrt\varepsilon)^\beta,DR\right)$-quasi-local. Also
\[
C(\varepsilon+\sqrt\varepsilon)^\beta \leq C(2\sqrt{\varepsilon})^\beta \leq 2C\varepsilon^{\beta/2}.
\]
Hence each $\chi_AP_{v^{(i)}}$ is $\left(2C\varepsilon^{\beta/2},DR\right)$-quasi-local.

Setting $w:=b/\|b\|$, \eqref{EQ:lower bdd of bA} shows that $\|w_A\|\geq \varepsilon^{\beta/4}$. Decomposing $w=\sum_{i=1}^r\lambda_iv^{(i)}$, the Cauchy-Schwarz inequality shows that
\[
\|w_A\| \leq \left(\sum_{i=1}^r\|v^{(i)}_A\|^2\right)^{1/2}.
\]
Hence, there exists $i_0 \in \{1,\cdots,r\}$ such that 
\[
\left\|v^{(i_0)}_A\right\| \geq \frac{\varepsilon^{\beta/4}}{\sqrt r} \geq \frac{\varepsilon^{\beta/4}}{\sqrt M}.
\]
Now Lemma~\ref{lem:large portion of norm} applied to $\chi_AP_{v^{(i_0)}}$ shows that $P_{v^{(i_0)}}$ is $\left( K\varepsilon^{\beta/4}, 2DR\right)$-quasi-local. Set $u^{(1)}:=v^{(i_0)}$. Then $\chi_YP_{u^{(1)}}$ is $(K\varepsilon^{\beta/4},2DR)$-quasi-local.

Therefore, in either \textbf{Case I} or \textbf{Case II}, we found $u^{(1)}\in U$ such that $\chi_YP_{u^{(1)}}$ is $(K\varepsilon^{\beta/4},2DR)$-quasi-local. For $U':=U \ominus\CC u^{(1)}$, $\chi_YP_{U'}$ is $(\varepsilon+K\varepsilon^{\beta/4},2DR)$-quasi-local. Since $\varepsilon < t_M \leq \left(\frac{t}{K+1}\right)^{4/\beta}$, we have $\varepsilon+K\varepsilon^{\beta/4} \leq (K+1)\varepsilon^{\beta/4} < t$. Applying the inductive hypothesis to $U'$, we obtain an orthonormal basis $\{u^{(2)}, \cdots, u^{(M)}\}$ of $U'$ such that each $\chi_YP_{u^{(i)}}$ is $\left(C((K+1)\varepsilon^{\beta/4})^\beta, D(2DR)\right)$-quasi-local for $i=2,\cdots,M$. Since
\[
C((K+1)\varepsilon^{\beta/4})^\beta \leq C(K+1)\varepsilon^{\beta^2/4} = C_M\varepsilon^{\beta_M}, \quad D(2DR)=2D^2R=D_MR
\]
and 
\[
K\varepsilon^{\beta/4} \leq C_M\varepsilon^{\beta_M}, \quad 2DR \leq 2D^2R=D_MR,
\]
we obtain that each $\chi_YP_{u^{(i)}}$ is $(C_M \varepsilon^{\beta_M}, D_M R)$-quasi-local for $i=1,\cdots, M$.
\end{proof}

We also need the following auxiliary lemma before we prove the general case.

\begin{lem}\label{lem:composition of quasi-lcoal}
  Let $(X,d)$ be a discrete metric space, $\H_0$ a Hilbert space and $T,S \in \B(\ell^2(X;\H_0))$. Assume that $T$ is $(\varepsilon_1, R_1)$-quasi-local and $S$ is $(\varepsilon_2, R_2)$-quasi-local for $\varepsilon_i>0$, $R_i>0$ and $i=1,2$. Then $TS$ is $(\|T\| \varepsilon_2 + \|S\|\varepsilon_1, R_1 + R_2)$-quasi-local.
\end{lem}

\begin{proof}
  Given $A, B\subset X$ with $d(A,B)>R_1 + R_2$, note that $d(\N_{R_1}(A), B) > R_2$ and $d(A, X \setminus \N_{R_1}(A)) > R_1$. Hence we have
  \begin{align*}
    \|\chi_{A}TS\chi_{B}\|& \leq\|\chi_{A}T\chi_{\N_{R_1}(A)}S\chi_{B}\|+\|\chi_{A}T\chi_{X \setminus \N_{R_1}(A)}S\chi_{B}\| \\
   &\leq \|T\|\cdot\|\chi_{\N_{R_1}(A)}S\chi_{B}\| + \|\chi_{A}T\chi_{X \setminus \N_{R_1}(A)}\|\cdot\|S\|\leq\|T\|\varepsilon_2 + \|S\|\varepsilon_1,
 \end{align*}
 which concludes the proof.
\end{proof}

\begin{thm}\label{thm:quantitative decomposition general case}
For each $M\in\NN$, there exist constants $s_M, \alpha_M\in(0,1)$ and $A_M, B_M\geq 1$, depending only on $M$, such that the following holds: Let $(X,d)$ be a discrete metric space, $Y\subseteq X$, $\H_0$ a Hilbert space, and let $T\in\B(\ell^2(X;\H_0))$ be a contraction of rank $M$. If $\chi_Y T$ is $(\varepsilon,R)$-quasi-local for some $\varepsilon\in (0,s_M)$ and $R>0$, then $(\ker T)^\perp$ has an orthonormal basis $\{u^{(1)},\ldots,u^{(M)}\}$ such that each $\chi_YTP_{u^{(i)}}$ is $(A_M\varepsilon^{\alpha_M},B_MR)$-quasi-local. 
\end{thm}

\begin{proof}
By adjusting the constants in Theorem~\ref{thm:quantitative decomposition} if necessary, we may assume that the same $t_M,\beta_M,C_M,D_M$ work in every dimension at most $M$.

Fix a discrete metric space $(X,d)$, a subset $Y \subseteq X$, a Hilbert space $\H_0$, and a contraction $T\in\B(\ell^2(X;\H_0))$ of rank $M$. Choose 
\[
\alpha_M:=\frac{\beta_M}{2(1+2M\beta_M)}, \quad K_M:=\frac{4M(4^M-1)(2M+2)^M}{3},
\]
and $s_M\in(0,1)$ sufficiently small such that $K_M \cdot s_M^{1-2M\alpha_M}<t_M$. 

Assume that $\chi_Y T$ is $(\varepsilon, R)$-quasi-local for some $\varepsilon \in (0,s_M)$. Then $S:=T^*\chi_YT$ is $(2\varepsilon,2R)$-quasi-local. Since $S$ has at most $M$ non-zero spectral values, we can take a closed interval $[a,b] \subseteq [\varepsilon^{2\alpha_M},2\varepsilon^{2\alpha_M}]$ such that $\sigma(S) \cap [a,b] = \emptyset$ and 
\[
b-a \geq \frac{\varepsilon^{2\alpha_M}}{2(M+1)}.
\]

Let $\lambda_1<\cdots<\lambda_r$ be the distinct non-zero points in \(\sigma(S)\) for some $r \leq M$, and set $\lambda_0=0$. Define \(f\) on \(\{\lambda_0,\ldots,\lambda_r\}\) by $f(\lambda_i)=0$ for $\lambda_i < a$ and $f(\lambda_i)=1$ if $\lambda_i > b$. Take $P:=f(S)$, and  we have
\begin{equation}\label{eq:simple-low}
\|\chi_Y T(1-P)\|\leq\sqrt2\,\varepsilon^{\alpha_M}.
\end{equation}
Now take the unique polynomial $p$ of degree at most $r$ such that $p(\lambda_i) = f(\lambda_i)$ for $i=0,1,\cdots,r$. By the Newton interpolation formula, we have
\[
p(t):=\sum_{k=1}^r f[\lambda_0,\ldots,\lambda_k] \prod_{j=0}^{k-1}(t-\lambda_j),
\]
where $f[\lambda_0,\ldots,\lambda_k]$ denotes the divided difference defined by $f[\lambda_i,\lambda_j] =\frac{f(\lambda_j)-f(\lambda_i)}{\lambda_j-\lambda_i}$ and recursively,
\[
f[\lambda_0,\ldots,\lambda_k] = \frac{f[\lambda_1,\ldots,\lambda_k]-f[\lambda_0,\ldots,\lambda_{k-1}]}{\lambda_k-\lambda_0}.
\]
By the choice of $f$, an induction gives $\bigl|f[\lambda_0,\ldots,\lambda_k]\bigr| \leq 2^{k-1}(b-a)^{-k}$ for each $k$. Writing $p=\sum_{j=1}^r c_j t^j$, we obtain that
\begin{equation}\label{EQ:Newton}
\sum_{j=1}^{r}|c_j| \leq \sum_{k=1}^{r} 2^{k-1}(b-a)^{-k}2^k \leq (b-a)^{-M} \cdot \sum_{k=1}^M 2^{2k-1} \leq \left(\frac{2M+2}{\varepsilon^{2\alpha_M}}\right)^{M} \cdot \frac{2(4^M-1)}{3}.
\end{equation}
Since $p=f$ on $\sigma(S)$, we have $P=p(S)=\sum_{j=1}^{r}c_jS^j$. 

Recall that $S$ is a contraction and is $(2\varepsilon,2R)$-quasi-local. Using Lemma~\ref{lem:composition of quasi-lcoal} several times, $S^j$ is $(2j\varepsilon, 2jR)$-quasi-local for every $1\leq j \leq r$. Combining with \eqref{EQ:Newton}, we obtain that for any $E,F\subseteq X$ with $d(E,F)>2MR$, we have
\[
\|\chi_E P\chi_F\|\leq \sum_{j=1}^{r}|c_j|\cdot \|\chi_E S^j\chi_F\| < 2M\varepsilon\cdot \sum_{j=1}^{r}|c_j| \leq K_M \cdot \varepsilon^{1-2M\alpha_M}.
\]
Hence, $P$ is $(K_M \varepsilon^{1-2M\alpha_M}, 2MR)$-quasi-local. 

If $P=0$, then \eqref{eq:simple-low} implies that $\|\chi_Y T\|\leq\sqrt2\,\varepsilon^{\alpha_M}$. Since $\dim(\ker T)^\perp=M$, any orthonormal basis $\{u^{(1)},\ldots,u^{(M)}\}$ of $(\ker T)^\perp$ satisfies $\|\chi_YTP_{u^{(i)}}\| \leq \sqrt{2}\varepsilon^{\alpha_M}$.

Now assume $P \neq 0$. Set $m:=\rank(P)\leq M$ and Theorem~\ref{thm:quantitative decomposition} provides an orthonormal basis $\{u^{(1)},\ldots,u^{(m)}\}$ of $\ran(P)$ such that each $P_{u^{(i)}}$ is
\[
\left(
C_MK_M^{\beta_M}
\varepsilon^{\beta_M(1-2M\alpha_M)},
\,2MD_MR
\right)\text{-quasi-local}.
\]
Since $\chi_YT$ is $(\varepsilon,R)$-quasi-local, Lemma~\ref{lem:composition of quasi-lcoal} shows that each $\chi_YTP_{u^{(i)}}$ is
\[
\left(
(1+C_MK_M^{\beta_M})\varepsilon^{\alpha_M}, (2MD_M+1)R
\right)\text{-quasi-local},
\]
where we use the fact $\beta_M(1-2M\alpha_M)>\alpha_M$ by the choice of $\alpha_M$.

Since $P$ is a spectral projection of $S=T^*\chi_YT$ corresponding to non-zero spectral values, we have
\[
\operatorname{Ran}(P)\subseteq(\ker S)^\perp
\subseteq(\ker T)^\perp.
\]
Complete $\{u^{(1)},\ldots,u^{(m)}\}$ to an orthonormal basis $\{u^{(1)},\ldots,u^{(M)}\}$ of $(\ker T)^\perp$. For $m<i\leq M$, we have
$u^{(i)}\in\ran(1-P)$, and hence~\eqref{eq:simple-low} gives
\[
\|\chi_YTP_{u^{(i)}}\|
\leq
\|\chi_YT(1-P)\|
\leq
\sqrt{2}\varepsilon^{\alpha_M}.
\]

Finally, we set
\[
A_M:=2+C_MK_M^{\beta_M}
\quad \text{and} \quad
B_M:=2MD_M+1.
\]
The preceding estimates show that $\chi_YTP_{u^{(i)}}$ is $(A_M\varepsilon^{\alpha_M},B_MR)$-quasi-local for each $i=1,\cdots,M$. This completes the proof.
\end{proof}

Now we use Theorem~\ref{thm:quantitative decomposition general case} to prove Theorem~\ref{thm:quantitative decomposition intro} as follows.

\begin{proof}[Proof of Theorem~\ref{thm:quantitative decomposition intro}]
By decomposing according to the block ranks, we may assume without loss of generality that $T=\SOT\text{-}\sum_{n\in \NN} T_n$ has block-rank $M$. Since $T$ is \((\varepsilon,R)\)-quasi-local, so is each \(T_n\). Applying Theorem~\ref{thm:quantitative decomposition general case} to \(T_n\), we obtain an orthonormal basis $\{u_n^{(1)},\ldots,u_n^{(M)}\}$ of \((\ker T_n)^\perp\) such that each \(T_nP_{u_n^{(j)}}\) is \((A_M\varepsilon^{\alpha_M},B_MR)\)-quasi-local. Define
\[
T^{(j)}:=\SOT\text{-}\sum_{n\in \NN} T_nP_{u_n^{(j)}} \quad \text{for} \quad j=1,\cdots,M.
\]
Then each \(T^{(j)}\) is a block-diagonal contraction of block-rank at most one and is \((2A_M\varepsilon^{\alpha_M},B_MR)\)-quasi-local. Since $T_n=\sum_{j=1}^{M}T_nP_{u_n^{(j)}}$ for each \(n\), we obtain $T=T^{(1)}+\cdots+T^{(M)}$. Replacing \(A_M\) by \(2A_M\), we conclude the proof. 
\end{proof}

\section{Quantitative approximation for block-rank-one case}\label{sec: rank one}

This section is devoted to the proof of Theorem~\ref{thm:quantitative version of rank one intro}. We start with the case of projections and $\H_0 = \CC$, stated as follows. Recall that for a discrete metric $(X,d)$, we denote $N_X(R)=\sup_{x\in X} |B(x,R)|$.

\begin{thm}\label{thm:quantitative version of rank one complex case}
 There exist functions
 \[
 L:(0,\frac{1}{10}) \longrightarrow (0, \infty) \quad \text{and} \quad h: (0,\frac{1}{10}) \times (0,+\infty)\times [1,+\infty) \longrightarrow (0,\infty)
 \]
such that for any discrete metric space $(X,d)$ of bounded geometry and a rank-one projection $P$ which is $(\varepsilon, R)$-quasi-local for some $\varepsilon \in (0,\frac{1}{10})$ and $R>0$, there exists $T \in \B(\ell^2(X))$ such that $\ppg(T) \leq h(\varepsilon, R, 4N_X(L(\varepsilon)R))$ and $\|T - P\| \leq 4\varepsilon$. Moreover, we can make $L$ non-increasing, $h$ non-increasing in the first variable and non-decreasing in the third.
\end{thm}

Theorem~\ref{thm:quantitative version of rank one complex case} is a quantitative version of \cite[Theorem 6.1]{LSZ23}. The proof follows the one for \cite[Theorem 6.1]{LSZ23}, and will be given in the next section. Let us first use it to prove the projection case in Theorem~\ref{thm:quantitative version of rank one intro}, stated as follows. We say that $T \in \B(\ell^2(X;\H_0))$ has \emph{rank locally at most one} if $rank(T_{x,y}) \leq 1$ for each $x,y\in X$.

\begin{cor}\label{cor:quantitative version of rank one}
The functions $L$ and $h$ in Theorem~\ref{thm:quantitative version of rank one complex case} satisfy the following: 
For any discrete metric space $(X,d)$ of bounded geometry, Hilbert space $\H_0$ and a rank-one projection $P$ which is $(\varepsilon, R)$-quasi-local for some $\varepsilon \in (0,\frac{1}{10})$ and $R>0$, there exists $T \in \B(\ell^2(X;\H_0))$ with rank locally at most one such that $\ppg(T) \leq h(\varepsilon, R, 4N_X(L(\varepsilon)R))$ and $\|T - P\| \leq 4\varepsilon$. 
\end{cor}

\begin{proof}
Let $(X,d)$ be a discrete metric space, $\H_0$ a Hilbert space and a unit vector $u \in \ell^2(X;\H_0)$. Write $u=\sum_{x\in X}\delta_x \otimes u_x$, where $\delta_x$ is the Dirac function at $x$ and $u_x \in \H_0$. Assume that $P_u$ is $(\varepsilon, R)$-quasi-local for some $\varepsilon \in (0,\frac{1}{10})$ and $R>0$.

 Following the proof of \cite[Lemma 6.5]{LSZ23}, we consider the \emph{uniformisation} vector $\tilde{u}$ of $u$ defined by $\tilde{u}:=\sum_{x\in X} \|u_x\| \delta_x \in \ell^2(X)$. Without loss of generality (see \cite[Lemma 4.5, Lemma 4.6]{LSZ23} and their proofs), we assume that $u$ has full support $X$. Define
 \[
   W: \ell^2(X) \longrightarrow \ell^2(X;\H_0) \quad \text{by} \quad \delta_x \mapsto \frac{\delta_x \otimes u_x}{\|u_x\|}\quad \text{for} \quad x\in X.
 \]
 It is routine to check that $W$ is an isometry with $W^*:\ell^2(X;\H_0) \to \ell^2(X)$ given by
 \[
   W^*(\delta_x \otimes \eta) := \frac{\langle \eta, u_x \rangle}{\|u_x\|} \delta_x \quad \text{for} \quad x\in X \quad \text{and} \quad \eta \in \H_0.
 \]
 Then we have $P_u = W P_{\tilde{u}} W^*$ and $P_{\tilde{u}} = W^* P_u W$.

 Following \cite[Lemma 4.3 and Lemma 6.4]{LSZ23}, we obtain that $P_{\tilde{u}}$ is also $(\varepsilon,R)$-quasi-local. Applying Theorem~\ref{thm:quantitative version of rank one complex case}, there exists $\tilde{T} \in \B(\ell^2(X))$ with $\ppg(\tilde{T}) \leq h(\varepsilon, R, 4N_X(L(\varepsilon)R))$ such that $\|\tilde{T} - P_{\tilde{u}}\| \leq 4\varepsilon$. Set $T:=W\tilde{T}W^* \in \B(\ell^2(X;\H_0))$. Since $W$ does not change propagation, it is clear that $\ppg(T) = \ppg(\tilde{T})$. On the other hand, the proof of \cite[Lemma 6.5]{LSZ23} indicates that
 \[
   T_{xy} \xi = \frac{\langle \xi, u_y \rangle \tilde{T}_{xy}}{\|u_x\| \cdot \|u_y\|} \cdot u_x \quad \text{for} \quad x,y \in X \quad \text{and} \quad \xi\in \H_0.
 \]
 Hence each $T_{x,y}$ has rank at most one. Finally, we have
 \[
   \|P_u - T\| = \|WP_{\tilde{u}}W^* - W \tilde{T}W^*\| \leq \|P_{\tilde{u}} - \tilde{T}\| \leq 4\varepsilon,
 \]
 which concludes the proof.
\end{proof}

We now pass to general rank-one contractions.

\begin{thm}\label{thm:quantitative version of rank one general case}
The functions $L$ and $h$ in Theorem~\ref{thm:quantitative version of rank one complex case} further satisfy the following: 
For any discrete metric space $(X,d)$ of bounded geometry, a Hilbert space $\H_0$ and a rank-one contraction $T \in \B(\ell^2(X;\H_0))$ which is $(\varepsilon, R)$-quasi-local for some $\varepsilon \in (0,\frac{1}{30})$ and $R>0$, there exists $S \in \B(\ell^2(X;\H_0))$ with rank locally at most one such that $\ppg(S) \leq h(\varepsilon, 2R, 4N_X(2L(3\varepsilon)R))$ and $\|S - T\| \leq 30\varepsilon$. 
\end{thm}

\begin{proof}
First, assume that $T \neq 0$ is a partial isometry. Then $P:=T^*T$ and $Q:=TT^*$ are projections, both of which are $(2\varepsilon, 2R)$-quasi-local due to Lemma~\ref{lem:composition of quasi-lcoal}. Consider the operator $G$ on $\ell^2(X;\H_0\oplus \H_0)$, defined by
\[
G:=\frac{1}{2}
\begin{pmatrix}
P&T^*\\
T&Q
\end{pmatrix}.
\]
This is a rank-one projection and $(3\varepsilon, 2R)$-quasi-local by direct calculations. Then Corollary~\ref{cor:quantitative version of rank one} shows that there exists $F \in \B(\ell^2(X;\H_0 \oplus \H_0))$ with rank locally at most one such that $\ppg(F) \leq h(3\varepsilon, 2R, 4N_X(2L(3\varepsilon)R))$ and $\|G-F\| \leq 12\varepsilon$. Let $S'$ be the lower-left corner of $F$ and set $S:=2S'$. Then $\ppg(S) \leq \ppg(F)$ and $\|S-T\| \leq 2\|G-F\| \leq 24\varepsilon$. 

Now consider a general contraction $T$. If $T=0$, the result holds trivially. So assume $T \neq 0$. Write $T = \lambda V$ for the rank-one partial isometry $V \in \B(\ell^2(X;\H_0))$ and $\lambda:=\|T\| \in (0,1]$. If $\lambda \leq 30\varepsilon$, simply take $S=0$. Otherwise, $V=\lambda^{-1}T$ is $(\frac{\varepsilon}{\lambda},R)$-quasi-local and $\frac{\varepsilon}{\lambda} < \frac{1}{30}$. The partial isometry case provides $F \in \B(\ell^2(X;\H_0))$ with 
\[
\ppg(F) \leq h\left(\frac{3\varepsilon}{\lambda}, 2R, 4N_X\left(2L\left(\frac{3\varepsilon}{\lambda}\right) R\right) \right)
\]
and $\|F-V\| \leq \frac{24\varepsilon}{\lambda}$.
Set $S:=\lambda F$. Then $\ppg(S) = \ppg(F)$ and we have
\[
\|S-T\| \leq \lambda \cdot \frac{24\varepsilon}{\lambda} = 24\varepsilon.
\]
Since $\frac{3\varepsilon}{\lambda} \geq 3\varepsilon \geq \varepsilon$, we have $L\left(\frac{3\varepsilon}{\lambda}\right) \leq L(3\varepsilon)$ and hence,
\[
\ppg(F) \leq h\left(\frac{3\varepsilon}{\lambda}, 2R, 4N_X\left(2L\left(\frac{3\varepsilon}{\lambda}\right) R\right) \right) \leq h(\varepsilon, 2R, 4N_X(2L(3\varepsilon)R)).
\]
Therefore, we conclude the proof.
\end{proof}

Based on Theorem~\ref{thm:quantitative version of rank one general case}, we prove Theorem~\ref{thm:quantitative version of rank one intro} as follows.

\begin{proof}[Proof of Theorem~\ref{thm:quantitative version of rank one intro}]
Since $T=(\mathrm{SOT})\text{-}\sum_{n\in\NN}T_n$ is $(\varepsilon,R)$-quasi-local, so is each $T_n$. 
Applying Theorem~\ref{thm:quantitative version of rank one general case} to each $T_n$, 
we obtain $S_n\in\B(\ell^2(X_n;\H_0))$ such that $\|T_n-S_n\|\leq 30\varepsilon$ and $\ppg(S_n)\leq h(\varepsilon,2R, 4N_{X_n}(2L(3\varepsilon)R))$.
Since each $S_n$ has rank locally at most one, $S:=(\mathrm{SOT})\text{-}\sum_{n\in\mathbb N}S_n$ is locally compact. Moreover, the uniformly bounded geometry of $\{X_n\}_{n\in\mathbb N}$ implies $S$ has finite propagation. Moreover, we have
\[
\|T-S\|
=
\sup_{n\in\mathbb N}\|T_n-S_n\|
\leq 30\varepsilon.
\]
This completes the proof.
\end{proof}

Finally, we provide a quantitative proof of Theorem~\ref{thm:Thm A in LZZ26}.

\begin{proof}[Proof of Theorem~\ref{thm:Thm A in LZZ26}]
Since $C^*(X;H_0)\subseteq C_q^*(X;H_0)$, it suffices to prove the reverse implication. 
By scaling, we assume that $T$ is a quasi-local contraction. Given $\delta>0$, we choose $\varepsilon>0$ such that $\varepsilon<s_M$, $A_M\varepsilon^{\alpha_M}<\frac{1}{30}$ and $30M A_M\varepsilon^{\alpha_M}<\delta$.

Since $T$ is quasi-local, there exists $R>0$ such that $T$ is $(\varepsilon,R)$-quasi-local. Theorem~\ref{thm:quantitative decomposition intro} provides a decomposition $T=T^{(1)}+\cdots+T^{(M)}$, where each $T^{(j)}$ has block-rank at most one and is $(A_M\varepsilon^{\alpha_M},B_MR)$-quasi-local. After discarding the zero blocks, we may assume that each \(T^{(j)}\) has block-rank one. Applying Theorem~\ref{thm:quantitative version of rank one intro} to each $T^{(j)}$, we obtain finite-propagation locally
compact operators $S^{(j)}$ such that $\|T^{(j)}-S^{(j)}\|
\leq
30A_M\varepsilon^{\alpha_M}$.
Set $S:=S^{(1)}+\cdots+S^{(M)}$. Then $S$ is locally compact with finite propagation and 
\[
\|T-S\|
\leq
30M A_M\varepsilon^{\alpha_M}
<\delta.
\]
Therefore, $T\in C^*(X;H_0)$.
\end{proof}

\section{Proof of Theorem~\ref{thm:quantitative version of rank one complex case}}\label{sec:proof of complex case rank one}

Here we prove Theorem~\ref{thm:quantitative version of rank one complex case}, following the outline of \cite[Theorem 6.1]{LSZ23} and \cite[Theorem 5.3]{LWZ25}. The idea is to translate the quasi-locality of a rank-one projection to the notion of measured asymptotic expanders introduced in \cite{LNSZ21}, and then apply structural results to approximate measured asymptotic expanders by measured expanders. Since we need a quantitative version, we have to trace parameters.

Let us fix a discrete metric space $(X,d)$ of bounded geometry and a unit vector $u\in \ell^2(X)$. Assume that $P_u$ is $(\varepsilon, R)$-quasi-local for some $\varepsilon \in (0,\frac{1}{10})$ and $R>0$.

\subsection{Quantitative expansion from quasi-locality}

Following \cite[Lemma 4.5 and Lemma 4.6]{LSZ23}, we assume that $u$ has full support $X$ without loss of generality. Define a probability measure $\mu$ on $X$ by $\mu(x):=|u(x)|^2$ for $x\in X$ and the following operator
\[
 U: \ell^2(X,\mu) \longrightarrow \ell^2(X) \quad \text{by} \quad (U\xi)(x):=u(x)\xi(x) \quad \text{for} \quad \xi \in \ell^2(X,\mu) \text{ and } x\in X.
\]
It is routine to check (see, \emph{e.g.}, \cite[Lemma 5.5]{LWZ25}) that $U$ is unitary and
\[
P:=U^*P_uU \in \B(\ell^2(X,\mu))
\]
is the orthogonal projection onto constant functions on $X$. Since $U$ has propagation zero, it is clear that $P \in \B(\ell^2(X,\mu))$ is also $(\varepsilon,R)$-quasi-local by Lemma~\ref{lem:composition of quasi-lcoal}.

\begin{lem}\label{lem:app MAE 1}
  For any $A \subseteq X$ with $4\varepsilon^2 \leq \mu(A) \leq \frac{1}{2}$, we have $\mu(\partial_R(A)) > \frac{1}{2} \mu(A)$.
\end{lem}

\begin{proof}
  For such $A \subseteq X$, \cite[Lemma 4.6]{LWZ25} shows that
  \[
   \mu(A) \cdot \mu(X \setminus \N_R(A)) = \|\chi_A P \chi_{X \setminus \N_R(A)}\|^2 < \varepsilon^2,
  \]
  which implies that $\mu(X \setminus \N_R(A)) < \frac{1}{4}$. Hence $\mu(\N_R(A)) > \frac{3}{4} \geq (1+\frac{1}{2}) \cdot \mu(A)$.
\end{proof}

Furthermore, we have:

\begin{lem}\label{lem:app MAE 2}
  For any $A \subseteq X$ with $4\varepsilon^2 \leq \mu(A) \leq \frac{7}{8}$, we have $\mu(\partial_R(A))> \frac{1}{28}\mu(A)$.
\end{lem}

\begin{proof}
  If $4\varepsilon^2 \leq \mu(A) \leq \frac{1}{2}$, Lemma~\ref{lem:app MAE 1} implies the result. If $\frac{1}{2} < \mu(A) \leq \frac{7}{8}$, set $B:=X \setminus \N_R(A)$. Then $\mu(B) \leq \mu(X \setminus A) < \frac{1}{2}$. We further divide into two cases.

  If $\mu(B) < \frac{1}{16}$, we have $\mu(\N_R(A)) > 1-\frac{1}{16} \geq (1+\frac{1}{14})\mu(A)$. If $\mu(B) \geq \frac{1}{16}$, Lemma~\ref{lem:app MAE 1} implies $\mu(\partial_R(B)) > \frac{1}{2}\mu(B)$. Since $\partial_R(B) \subseteq \partial_R(A)$, then $\mu(\partial_R(A)) > \frac{1}{2}\mu(B) \geq \frac{1}{28}\mu(A)$.
\end{proof}

Take $L \in \NN$ such that $(1+\frac{1}{28})^L \geq \frac{1}{4\varepsilon^2}$. We have the following:

\begin{lem}\label{lem:app MAE 3 domain}
  Given $Y \subseteq X$ with $\mu(Y) \geq \frac{1}{2}$ and $A \subseteq Y$ with $8\varepsilon^2\mu(Y) \leq \mu(A) \leq \frac{1}{2}\mu(Y)$, we have $\mu\left(\partial_{LR}^Y(A)\right) > \frac{1}{2} \mu(A)$.
\end{lem}

\begin{proof}
  For such $Y$ and $A$, we claim that $\mu(\N_{LR}(A)) > 1-\frac{1}{4}\mu(Y)$. If not, we have
  \[
  4\varepsilon^2 \leq 8\varepsilon^2 \mu(Y) \leq  \mu(A) \leq \mu(\N_{iR}(A)) \leq 1-\frac{1}{4}\mu(Y) \leq \frac{7}{8} \quad \text{for} \quad i=1,2,\cdots, L.
  \]
  Applying Lemma~\ref{lem:app MAE 2} to $\N_{iR}(A)$ for $i=0,1,\cdots,L-1$, we obtain
  \[
   \mu(\N_{LR}(A)) > \left(1+\frac{1}{28}\right)^L \cdot \mu(A) \geq \left(1+\frac{1}{28}\right)^L \cdot 4\varepsilon^2 \geq 1,
  \]
  which leads to a contradiction. Hence we have $\mu(\N_{LR}(A)) > 1-\frac{1}{4}\mu(Y)$ and then,
  \[
   \mu\left(\N_{LR}(A) \cap Y \right) > 1-\frac{1}{4}\mu(Y) - \mu(X \setminus Y) = \frac{3}{4}\mu(Y) \geq \left(1+\frac{1}{2}\right)\mu(A).
  \]
  This concludes the proof.
\end{proof}

Recall a basic fact from coarse geometry (see, \emph{e.g.}, the proof of \cite[Lemma 12.2.3]{WY20}) that the set $\left\{(x,y) \in X: d(x,y) \leq LR\right\}$ can be written as a union of at most $2N_X(LR)$ many partial translations. Taking their inverses as well, we can write

\begin{equation}\label{EQ:app decomposition}
  \left\{(x,y) \in X: d(x,y) \leq LR\right\} = \bigcup_{i=1}^{N} K_i \quad \text{for} \quad N:=4N_X(LR),
\end{equation}
where $K_i$ is a partial translation so that for each $i$ there exists $j$ such that $K_i^{-1} = K_j$. This symmetry will ensure the reversibility of the Markov kernel constructed below.
Each $K_i$ corresponds to a bijection $\tau_i: D_i \to R_i$, whose graph coincides with $K_i$. Also set $\R_i:D_i \to (0,\infty)$ by $\R_i(x):=\frac{\mu(\tau_i(x))}{\mu(x)}$ for $x\in D_i$.

\begin{lem}\label{lem:app domain of expansion}
 There exists $Y \subseteq X$ with $\mu(Y) > (1-8\varepsilon^2)(1-\varepsilon^2)$ such that for any $A \subseteq Y$ with $0< \mu(A) \leq \frac{1}{2}\mu(Y)$, we have $\mu\left(\partial_{LR}^Y(A)\right) > \frac{1}{4}\mu(A)$. Moreover, we can require $\frac{\varepsilon^2}{N} \leq \R_i(x) \leq \frac{N}{\varepsilon^2}$ for any $i=1,\cdots, N$ and $x\in D_i \cap Y \cap \tau_i^{-1}(Y)$.
\end{lem}

\begin{proof}
  Firstly, we set
  \[
   Z:=\bigcup_{i=1}^{N} \left\{x\in D_i: \R_i(x) > \frac{N}{\varepsilon^2}\right\}.
  \]
  Then $\mu(Z) < \varepsilon^2$ and hence $\mu(X \setminus Z) > 1-\varepsilon^2 \geq \frac{1}{2}$. For $x\in X \setminus Z$, if $x \in D_i$ for some $i$ such that $\tau_i(x) \in X \setminus Z$, then $\R_i(x) \leq \frac{N}{\varepsilon^2}$. By symmetry, we also have $\R_i(x) \geq \frac{\varepsilon^2}{N}$.

  Take $F$ to be a maximal subset in $X \setminus Z$ satisfying both $\mu(F) \leq \frac{1}{2}\mu(X \setminus Z)$ and $\mu\left(\partial_{LR}^{X\setminus Z}(F)\right) \leq \frac{1}{2} \mu(F)$, whose existence is guaranteed by Zorn's lemma. By Lemma~\ref{lem:app MAE 3 domain}, we have $\mu(F) < 8\varepsilon^2 \mu(X \setminus Z)$. Set $Y:=X \setminus (Z \sqcup F)$. Then
  \begin{equation*}\label{EQ:app domain Y}
   \mu(Y) > (1-8\varepsilon^2)\mu(X \setminus Z) > (1-8\varepsilon^2)(1-\varepsilon^2).
  \end{equation*}

  Now we show that $Y$ satisfies the requirement. 
Given a subset $A \subseteq Y$ with $0< \mu(A) \leq \frac{1}{2}\mu(Y)$, we divide into two cases:

  \textbf{Case I.} $\mu(A) \leq \frac{1}{2}\mu(X \setminus Z) - \mu(F)$. Since $\partial_{LR}^{X \setminus Z}(A \sqcup F) \subseteq \partial_{LR}^{X \setminus Z}(F) \cup \partial_{LR}^{Y}(A)$, we obtain
  \[
   \mu\left(\partial_{LR}^{X \setminus Z}(A \sqcup F)\right) \leq \mu\left( \partial_{LR}^{X \setminus Z}(F) \right) + \mu\left( \partial_{LR}^{Y}(A) \right) \leq \frac{1}{2}\mu(F) + \mu\left( \partial_{LR}^{Y}(A) \right).
  \]
  On the other hand, the maximality of $F$ implies that
  \[
   \mu\left( \partial_{LR}^{X \setminus Z}(A \sqcup F) \right) > \frac{1}{2} \mu(A\sqcup F) = \frac{1}{2} \mu(A) + \frac{1}{2}\mu(F).
  \]
  Combining them together, we obtain that $\mu\left( \partial_{LR}^{Y}(A) \right) > \frac{1}{2}\mu(A)$.

  \textbf{Case II.} $\mu(A) > \frac{1}{2}\mu(X \setminus Z) - \mu(F)$. We have
  \[
   \frac{1}{2} \mu(X \setminus Z) \geq \frac{1}{2} \mu(Y) \geq \mu(A) > \left(\frac{1}{2} - 8\varepsilon^2\right)\mu(X \setminus Z) \geq 8\varepsilon^2\mu(X \setminus Z).
  \]
  Then Lemma~\ref{lem:app MAE 3 domain} implies that $\mu\left(\partial_{LR}^{X \setminus Z}(A)\right) > \frac{1}{2} \mu(A)$. Since
  \[
   \mu\left(\partial_{LR}^{Y}(A)\right) \geq \mu\left(\partial_{LR}^{X \setminus Z}(A)\right) - \mu(F)  > \frac{1}{2}\mu(A) - \mu(F)
  \]
  and
  \[
   \mu(A) > \left(\frac{1}{2} - 8\varepsilon^2\right)\mu(X \setminus Z) > \frac{2}{5}\mu(X \setminus Z),
  \]
  we obtain
  \[
   \mu\left(\partial_{LR}^{Y}(A)\right) > \frac{1}{2}\mu(A) - \mu(F) > \frac{1}{2}\mu(A) - 8\varepsilon^2\mu(X \setminus Z) > \frac{1}{2}\mu(A) - 8\varepsilon^2 \cdot \frac{5}{2}\mu(A) > \frac{1}{4} \mu(A).
  \]
  Hence we conclude the proof.
\end{proof}

\subsection{Markov operators and finite-propagation approximation}

We use Markov operators (see \cite{Rev75} and also \cite{LVZ23b, LWZ25} for details) and follow the outline of \cite[Section 3]{LWZ25}. With the decomposition (\ref{EQ:app decomposition}) and $Y \subseteq X$ from Lemma~\ref{lem:app domain of expansion}, for $x\in Y$ we set
\begin{equation*}\label{EQ:app Markov auxiliary}
 K_{Y,x}:=\{i=1,\cdots, N: x\in D_i \text{~and~} \tau_i(x) \in Y\} \quad \text{and} \quad \sigma_{Y}(x)=\sum_{i \in K_{Y,x}} \R_i(x)^{\frac{1}{2}}.
\end{equation*}

\begin{lem}\label{lem:app sigma geq 1}
For any $x\in Y$, we have $1 \leq \sigma_Y(x) \leq \frac{N^{3/2}}{\varepsilon}$.
\end{lem}

\begin{proof}
  Since $\bigsqcup_{i=1}^N K_i$ covers the diagonal in $X \times X$, there exists $i \in \{1,\cdots, N\}$ such that $(x,x) \in K_i$ for any $x\in Y$. Hence $\sigma_Y(x) \geq \R_i(x)^{\frac{1}{2}} =1$. On the other hand, the final assertion of Lemma~\ref{lem:app domain of expansion} implies
  \[
    \sigma_Y(x) \leq \sum_{i=1}^N \sqrt{\R_i(x)} \leq N\sqrt{\frac{N}{\varepsilon^2}} = \frac{N^{3/2}}{\varepsilon}.
  \]
  Hence we conclude the proof.
\end{proof}

Consider the \emph{normalised Markov kernel}\footnote{Usually the Markov kernel is defined to be a map from $Y \times \mathcal{S}$ to $[0,1]$, where $\mathcal{S}$ is a $\sigma$-algebra on $Y$. In the setting above, we can equivalently define the kernel as a function on $Y \times Y$.} $\Pi_Y: Y \times Y \to [0,1]$ defined by
\begin{equation*}\label{EQ:Markov for K}
 \Pi_{Y}(x,y)=\frac{1}{\sigma_{Y}(x)}\sum_{i\in K_{Y,x}} \R_i(x)^{\frac{1}{2}}\delta_{\tau_{i}(x)}(y) \quad \text{for} \quad  x, y\in Y.
\end{equation*}
Here $\delta_y$ is the Dirac delta measure on $y$. Define the measure $\tilde{\mu}_Y$ on $Y$ by
\begin{equation*}\label{EQ:reservsing measure}
 \tilde{\mu}_{Y}(x):= \sigma_{Y}(x) \cdot \mu(x) \quad \text{for} \quad x\in Y.
\end{equation*}
The symmetry in (\ref{EQ:app decomposition}) gives:
\begin{equation}\label{EQ:app reversible}
  \Pi_Y(x,y) \tilde{\mu}_Y(x) = \Pi_Y(y,x) \tilde{\mu}_Y(y) \quad \text{for any} \quad x,y\in Y.
\end{equation}
In this case, $\Pi_Y$ is called a \emph{reversible} Markov kernel with \emph{reversing measure} $\tilde{\mu}_Y$.

Recall from \cite[Definition 2.11]{LWZ25} that the Cheeger constant for $\Pi_Y$ on $(Y,\tilde{\mu}_{Y})$ is
\[
 \kappa_Y:=\inf\left\{\frac{|\partial_{\Pi_Y} (A)|_{\tilde{\mu}_{Y}}}{\tilde{\mu}_{Y}(A)}: A\subseteq Y \text{~with~} 0<\tilde{\mu}_{Y}(A)\leq\frac 12 \tilde{\mu}_{Y}(Y)\right\},
\]
where $|\partial_{\Pi_Y} (A)|_{\tilde{\mu}_{Y}} := \sum_{x\in A} \sum_{y\in Y \setminus A} \Pi_Y(x,y) \tilde{\mu}_Y(x)$.

The following result comes from \cite[Lemma 3.17(1)]{LWZ25} (see also \cite[Lemma 3.14]{LVZ23b}). For the reader's convenience, we provide a proof here.

\begin{lem}\label{lem:app lower bound for kappa}
  We have $\kappa_Y \geq \frac{\varepsilon^2}{4N^2}$.
\end{lem}

\begin{proof}
  Given $A\subseteq Y$ with $0<\tilde{\mu}_{Y}(A)\leq\frac 12 \tilde{\mu}_{Y}(Y)$, we have
  \begin{align*}
    |\partial_{\Pi_Y} (A)|_{\tilde{\mu}_{Y}} &= \sum_{x\in A} \sum_{y\in Y \setminus A} \left(\frac{1}{\sigma_{Y}(x)}\sum_{i\in K_{Y,x}} \R_i(x)^{\frac{1}{2}}\delta_{\tau_{i}(x)}(y) \right) \cdot \sigma_{Y}(x) \mu(x) \\
    &= \sum_{x\in A} \sum_{i:\tau_i(x) \in Y\setminus A} \sqrt{\mu(\tau_i(x))\mu(x)} \geq \frac{\varepsilon}{\sqrt{N}} \cdot \mu(\partial_{LR}^Y(A)),
  \end{align*}
  where the last inequality follows from Lemma~\ref{lem:app domain of expansion}. Also Lemma~\ref{lem:app sigma geq 1} shows
  \[
   \mu(A) \leq \tilde{\mu}_Y(A) \leq \frac{N^{3/2}}{\varepsilon} \mu(A).
  \]

  If $\mu(A) \leq \frac{1}{2}\mu(Y)$, Lemma~\ref{lem:app domain of expansion} directly implies
  \[
   \frac{|\partial_{\Pi_Y} (A)|_{\tilde{\mu}_{Y}}}{\tilde{\mu}_{Y}(A)} \geq \frac{\varepsilon^2}{N^2} \cdot \frac{\mu(\partial_{LR}^Y(A))}{\mu(A)} > \frac{\varepsilon^2}{4N^2}.
  \]
  If $\mu(A) > \frac{1}{2}\mu(Y)$, apply the above to $Y \setminus A$ and we obtain
  \[
  |\partial_{\Pi_Y} (A)|_{\tilde{\mu}_{Y}} = |\partial_{\Pi_Y} (Y \setminus A)|_{\tilde{\mu}_{Y}} > \frac{\varepsilon^2}{4N^2} \cdot \tilde{\mu}_Y(Y\setminus A) \geq \frac{\varepsilon^2}{4N^2} \cdot \tilde{\mu}_Y(A),
  \]
  where the first equality follows from (\ref{EQ:app reversible}). Hence we finish the proof.
\end{proof}

Define the associated \emph{Markov operator} $\P \in \B(\ell^2(Y,\tilde{\mu}_{Y}))$ by
\begin{equation*}\label{EQ:Markov operator defn}
 \P f(x):=  \sum_{y\in Y} f(y)\Pi_Y(x,y) \quad \text{for} \quad f\in \ell^2(Y,\tilde{\mu}_{Y}) \quad \text{and} \quad x\in Y.
\end{equation*}
By (\ref{EQ:app reversible}), $\P$ is a bounded self-adjoint operator with spectrum contained in $[-1,1]$. Denote $\ell^2_0(Y,\tilde{\mu}_{Y})$ the orthogonal complement of constant functions in $\ell^2(Y,\tilde{\mu}_{Y})$. It follows from \cite[Theorem 2.1]{LS88} together with Lemma~\ref{lem:app lower bound for kappa} that the supremum $\lambda$ of the spectrum of $\P|_{\ell^2_0(Y,\tilde{\mu}_{Y})}$ satisfies:
\begin{equation}\label{EQ:spectral gap control}
1- \lambda \geq \frac{\kappa_Y^2}{2} \geq \frac{\varepsilon^4}{32N^4}.
\end{equation}

Consider the embedding $I: \ell^2(Y,\tilde{\mu}_Y) \longrightarrow \ell^2(X,\mu)$ simply by extending each function in $\ell^2(Y,\tilde{\mu}_Y)$ by zero on $X \setminus Y$. Since $\tilde{\mu}_Y$ is equivalent to $\mu_Y$ and $\tilde{\mu}_Y \geq \mu_Y$ by Lemma~\ref{lem:app sigma geq 1}, this is well-defined and $\|I\| \leq 1$.
Denote the map
\[
 \Ad_I: \B(\ell^2(Y,\tilde{\mu}_Y)) \longrightarrow \B(\ell^2(X,\mu)), \quad T \mapsto I \circ T \circ I^* \quad \text{for} \quad T \in \B(\ell^2(Y,\tilde{\mu}_Y)).
\]
Define $P_Y \in \B(\ell^2(X,\mu))$ by $P_Y(\xi):=\frac{\langle \xi, \chi_Y \rangle}{\mu(Y)}\chi_Y$ for $\xi \in \ell^2(X,\mu)$, which is the orthogonal projection in $\B(\ell^2(X,\mu))$ onto $\CC\chi_Y$. Denote by $\tilde{P}_Y \in \B(\ell^2(Y,\tilde{\mu}_Y))$ the orthogonal projection onto constant functions on $Y$ in $\ell^2(Y,\tilde{\mu}_Y)$. Direct calculations show that
\begin{equation}\label{EQ:relation between P and tildeP}
 \Ad_I(\tilde{P}_Y) = \frac{\mu(Y)}{\tilde{\mu}_Y(Y)} \cdot P_Y.
\end{equation}

Recall from (\ref{EQ:spectral gap control}) that $\frac{1}{2} \chi_{Y} + \frac{1}{2}\P$ has spectrum contained in $[0,1-\frac{\varepsilon^4}{64N^4}]\cup\{1\}$ Therefore, for any $k \in \NN$ we have
\[
\left\|\left(\frac{1}{2}\chi_{Y}+\frac{1}{2}\P\right)^k-\tilde{P}_Y\right\| \leq \left(1-\frac{\varepsilon^4}{64N^4}\right)^k.
\]
Applying $\Ad_I$ and using (\ref{EQ:relation between P and tildeP}), we obtain
\begin{small}
\begin{align}\label{EQ: app for proj}
\left\|\frac{\tilde{\mu}_Y(Y)}{\mu(Y)} \cdot\Ad_I\left[\left(\frac{1}{2}\chi_{Y}+\frac{1}{2}\P\right)^k\right]- P_Y\right\| \leq \frac{\tilde{\mu}_Y(Y)}{\mu(Y)} \cdot \left(1-\frac{\varepsilon^4}{64N^4}\right)^k
\leq \frac{N^{3/2}}{\varepsilon} \cdot \left(1-\frac{\varepsilon^4}{64N^4}\right)^k,
\end{align}
\end{small}
where the second inequality comes from Lemma~\ref{lem:app sigma geq 1}.
Choose $K \in \NN$ such that
\begin{equation}\label{EQ:estimate in pf of main thm}
\left(1-\frac{\varepsilon^4}{64N^4}\right)^K \leq \frac{\varepsilon^2}{N^{3/2}}.
\end{equation}
Also direct calculations show that
\[
\|P_Y - P\| \leq \sqrt{\mu(X \setminus Y)}  = \sqrt{1-(1-8\varepsilon^2)(1-\varepsilon^2)} \leq 3\varepsilon.
\]
Combining them together, we obtain
\begin{align*}\label{EQ:estimate in pf of main thm3}
\Big\|\frac{\tilde{\mu}_Y(Y)}{\mu(Y)}  \cdot\Ad_I\left[\left(\frac{1}{2}\chi_{Y}+\frac{1}{2}\P\right)^K\right]- P\Big\| & \leq \Big\|\frac{\tilde{\mu}_Y(Y)}{\mu(Y)} \cdot\Ad_I\left[\left(\frac{1}{2}\chi_{Y}+\frac{1}{2}\P_Y\right)^K\right]- P_Y\Big\| + \|P_Y - P\| \nonumber \\
& \leq \frac{N^{3/2}}{\varepsilon} \cdot \left(1-\frac{\varepsilon^4}{64N^4}\right)^K + 3\varepsilon \\
& \leq \varepsilon + 3\varepsilon =4 \varepsilon.
\end{align*}

Finally, note that the following operator
\[
\tilde{T}:=\frac{\tilde{\mu}_Y(Y)}{\mu(Y)}  \cdot\Ad_I\left[\left(\frac{1}{2}\chi_{Y}+\frac{1}{2}\P\right)^K\right] \quad \in \quad \B(\ell^2(X,\mu))
\]
has propagation at most $KLR$. Setting $T:=U\tilde{T}U^* \in \B(\ell^2(X))$, we have
\[
\|T - P_u\| = \|U\tilde{T}U^* - UPU^*\| \leq 4\varepsilon
\]
and $T$ has the same propagation as that of $\tilde{T}$.

In conclusion, we proved: taking $L \in \NN$ with $(1+\frac{1}{28})^L \geq \frac{1}{4\varepsilon^2}$, $N:=4N_X(LR)$ and $K$ satisfying (\ref{EQ:estimate in pf of main thm}), there exists $T \in \B(\ell^2(X))$ with propagation at most $KLR$ such that $\|T-P_u\| \leq 4\varepsilon$. Therefore, we take
\[
L:(0,\frac{1}{10}) \longrightarrow (0,+\infty), \quad \varepsilon \mapsto \left\lceil \frac{-\log(4\varepsilon^2)}{\log(29/28)} \right\rceil
\]
and
\[
h:(0,\frac{1}{10}) \times (0,+\infty) \times [1,+\infty) \longrightarrow (0,+\infty) \quad \text{by} \quad  (\varepsilon, R, N) \mapsto K(\varepsilon,N)L(\varepsilon)R,
\]
where
\[
K(\varepsilon,N)
=
\left\lceil
\frac{64N^4}{\varepsilon^4}
\log\frac{N^{3/2}}{\varepsilon^2}
\right\rceil.
\]
Here we use the elementary estimate that $1-t \leq e^{-t}$ for any $t\in (0,1)$ to verify that $K$ works. It is clear from these formulas that $L$ is non-increasing, while $K(\cdot,N)$ is non-increasing and $K(\varepsilon,\cdot)$ is non-decreasing. Hence $h$ is non-increasing in its first variable and non-decreasing in its third variable. These functions fulfill the task and we complete the proof.

\bibliographystyle{plain}
\bibliography{reference_quantitative}

\end{document}